\documentclass{amsart}

\usepackage{amsmath}
\usepackage{amssymb}
\usepackage{amsthm}
\usepackage{mathrsfs}

\DeclareMathOperator{\Spec}{Spec}

\DeclareMathOperator{\Aut}{Aut}
\DeclareMathOperator{\Hom}{Hom}

\newtheorem{thm}{Theorem}[section]
\newtheorem{prop}[thm]{Proposition}

\newtheorem{lem}[thm]{Lemma}

\newtheorem{rem}[thm]{Remark}
\newtheorem{eg}[thm]{Example}
\newtheorem*{thmn}{Theorem}

\newcommand{\F}{\mathbb{F}}

\begin{document}

\title
[Finite flat models of constant group schemes of rank two]%
{Finite flat models of constant group schemes\\ of rank two}
\author{Naoki Imai}
\address{Graduate School of Mathematical Sciences, 
the University of Tokyo, 3-8-1 Komaba, Meguro-ku, Tokyo 153-8914, Japan.}
\email{naoki@ms.u-tokyo.ac.jp}


\maketitle

\begin{abstract}
We calculate the number of 
the isomorphism class of 
the finite flat models 
over the ring of integers of 
an absolutely ramified $p$-adic field 
of constant group schemes of rank two 
over finite fields, 
by counting the rational points of 
a moduli space of finite flat models. 
\end{abstract}

\section*{Introduction}
Let $K$ be a totally ramified extension 
of degree $e$ over $\mathbb{Q}_p$ for $p>2$, and $\F$ be 
a finite field of characteristic $p$. 
We consider the constant group scheme $C_{\F}$ 
over $\Spec K$ 
of the two-dimensional vector space 
over $\F$. 
A finite flat models of $C_{\F}$ is 
a pair 
$(\mathcal{G},C_{\F} \xrightarrow{\sim} \mathcal{G}_K )$ 
such that $\mathcal{G}$ is a finite flat group scheme 
over $\mathcal{O}_K$ with a structure 
of an $\F$-vector space. 
Here $\mathcal{G}_K$ is the generic fiber of $\mathcal{G}$, 
and $C_{\F} \xrightarrow{\sim} \mathcal{G}_K$ is 
an isomorphism of group schemes over 
$\Spec K$ that is 
compatible with the action of $\F$. 
Let $M(C_{\F} ,K)$ be the set of 
the isomorphism class of 
the finite flat models of $C_{\F}$. 
If $e <p-1$, then 
$M(C_{\F} ,K)$ is one-point set 
by \cite[Theorem 3.3.3]{Ray}. 
However, if the ramification is big, 
there are surprisingly many finite flat models. 
In this paper, 
we calculate 
the number of the isomorphism class 
of the finite flat models of $C_{\F}$, 
that is, 
$|M(C_{\F} ,K)|$. 
The main theorem is the following. 

\begin{thmn}
Let $q$ be the cardinality of $\F$. 
Then we have 
\[
 |M(C_{\F} ,K)| = 
 \sum_{n \geq 0} (a_n +a_n ')q^n . 
\]
Here $a_n$ and $a_n '$ are defined as 
in the following. 

We express $e$ and $n$ by 
\[
 e=(p-1)e_0 +e_1 ,\ 
 n=(p-1)n_0 +n_1 =(p-1)n_0 '+n_1 ' +e_1
\]
such that $e_0 ,n_0 ,n_0 '\in \mathbb{Z}$ 
and $0\leq e_1 ,n_1 ,n_1 '\leq p-2$. 
Then 
\begin{align*}
 a_n =& \max \bigl\{
 e_0 -(p+1)n_0 -n_1 -1,0 \bigr\} & 
 &\textrm{ if } 
 n_1 \neq 0,1, \\ 
 a_n =& \max \bigl\{
 e_0 -(p+1)n_0 -n_1 -1,0 \bigr\} & & \\ 
 &+\max \bigl\{
 e_0 -(p+1)n_0 -n_1 +1,0 \bigr\} & 
 &\textrm{ if } 
 n_1 = 0,1, \\ 
\intertext{and} 
 a_n ' =& \max \bigl\{
 e_0 -e_1 -(p+1)n_0 '-n_1 '-2,0 \bigr\} & 
 &\textrm{ if } 
 n_1 '\neq 0,1, \\ 
 a_n '=& \max \bigl\{
 e_0 -e_1 -(p+1)n_0 '-n_1 '-2,0 \bigr\} & & \\ 
 &+\max \bigl\{
 e_0 -e_1 -(p+1)n_0 '-n_1 ',0 \bigr\} & 
 &\textrm{ if } 
 n_1 '= 0,1 
\end{align*}
except in the case where $n=0$ and 
$e_1 =p-2$, in which case 
we put $a_0 '=e_0$. 
\end{thmn}
In the above theorem, 
we can easily check that 
$|M(C_{\F} ,K)| =1$ if $e<p-1$. 

\subsection*{Acknowledgment}
The author is grateful to his advisor Takeshi Saito 
for his careful reading of an earlier version of this paper 
and for his helpful comments. 
He would like to thank the referee 
for a careful reading of this paper 
and suggestions for improvements. 

\subsection*{Notation}
Throughout this paper, we use the following notation.
Let $p>2$ be a prime number, and $K$ be
a totally ramified extension of $\mathbb{Q}_p$ of degree $e$. 
The ring of integers of $K$ is denoted by $\mathcal{O}_K$, 
and the absolute Galois group of $K$ is 
denoted by $G_K$. 
Let $\F$ be a finite field of characteristic $p$.
The formal power series ring of $u$ over $\F$ is
denoted by $\F [[u]]$, 
and its quotient field is denoted by $\F ((u))$. 
Let $v_u$ be the valuation of 
$\F ((u))$ normalized by $v_u (u)=1$, 
and we put $v_u (0) =\infty$. 
For $x \in \mathbb{R}$, the greatest integer 
less than or equal to $x$ is denoted by $[x]$. 
\section{Preliminaries}
To calculate the number of 
finite flat models of $C_{\F}$, 
we use the moduli spaces of 
finite flat models constructed by Kisin in \cite{Kis}. 

Let $V_{\F}$ be the two-dimensional 
trivial representation of $G_K$ over $\F$. 
The moduli space of finite flat models of $V_{\F}$, 
which is denoted by $\mathscr{GR}_{V_{\F},0}$, 
is a projective scheme over $\F$. 
An important property of $\mathscr{GR}_{V_{\F},0}$ 
is the following Proposition.

\begin{prop}\label{property}
For any finite extension $\F'$ of $\F$, 
there is a natural bijection between 
the set of isomorphism classes of finite flat models 
of $V_{\F'} =V_{\F} \otimes _{\F} \F'$ 
and $\mathscr{GR}_{V_{\F},0} (\F')$.
\end{prop}

\begin{proof}
This is \cite[Corollary 2.1.13]{Kis}.
\end{proof}

By Proposition \ref{property}, to calculate the number of finite flat models, 
it suffices to count the number of the 
$\F$-rational points of $\mathscr{GR}_{V_{\F},0}$. 

Let $\mathfrak{S} = \mathbb{Z}_p [[u]] $, and $\mathcal{O}_{\mathcal{E}}$ be 
the $p$-adic completion of $\mathfrak{S} [1/u]$. 
There is an action of $\phi$ on 
$\mathcal{O}_{\mathcal{E}}$ 
determined by identity on $\mathbb{Z}_p$ 
and $u \mapsto u^p$. 
We choose elements $\pi _m \in \overline{K}$ such that
$\pi _0 = \pi$ and $\pi ^p _{m+1} = \pi _m$ for $m \geq 0$, 
and put $K_{\infty} = \bigcup _{m\geq 0} K(\pi _m )$. 
Let $\Phi {\mathrm{M}}_{\mathcal{O}_{\mathcal{E}},\F}$ 
be the category of finite 
$\mathcal{O}_{\mathcal{E}} \otimes_{\mathbb{Z}_p} \F$-modules 
$M$ equipped with $\phi$-semi-linear map 
$M \to M$ such that the induced 
$\mathcal{O}_{\mathcal{E}} \otimes_{\mathbb{Z}_p} \F$-linear 
map $\phi ^* (M) \to M$ is an isomorphism. 
We take the $\phi$-module 
$M_{\F} \in \Phi {\mathrm{M}}_{\mathcal{O}_{\mathcal{E}},\F}$ 
that corresponds to the $G_{K_{\infty}}$-representation 
$V_{\F} (-1)$. Here $(-1)$ denotes 
the inverse of the Tate twist. 

The moduli space $\mathscr{GR}_{V_{\F},0}$ is described via 
the Kisin modules as in the following.

\begin{prop}\label{description}
For any finite extension $\F'$ of $\F$, 
the elements of $\mathscr{GR}_{V_{\F},0}(\F')$ 
naturally correspond to free
$\F'[[u]]$-submodules 
$\mathfrak{M}_{\F'} \subset M_\F \otimes _{\F} \F'$ 
of rank $2$ that satisfy 
$u^e \mathfrak{M}_{\F'} \subset 
(1\otimes \phi ) \bigl( \phi ^* (\mathfrak{M}_{\F'}) \bigr)
\subset \mathfrak{M}_{\F'}$.
\end{prop}

\begin{proof}
This follows from the construction 
of $\mathscr{GR}_{V_{\F},0}$ in 
\cite[Corollary 2.1.13]{Kis}. 
\end{proof}

By Proposition \ref{description}, 
we often identify a point of 
$\mathscr{GR}_{V_{\F},0}(\F')$ with 
the corresponding finite free 
$\F' [[u]]$-module. 

For $A \in GL_2 \bigl( \F ((u)) \bigr)$, 
we write $M_{\F} \sim A$ 
if there is a basis $\{ e_1 ,e_2 \}$ 
of $M_{\F}$ over $\F ((u))$ 
such that 
$\phi 
\begin{pmatrix} e_1 \\ e_2 \end{pmatrix}
 = A 
\begin{pmatrix} e_1 \\ e_2 \end{pmatrix}$.
We use the same notation for any sublattice 
$\mathfrak{M} _{\F} \subset M_{\F}$ similarly. 

Finally, 
for any sublattice 
$\mathfrak{M} _{\F} \subset M_{\F}$ with a chosen 
basis $\{ e_1 ,e_2 \}$ and 
$B \in GL _2 \bigl( \F((u)) \bigr)$, 
the module 
generated by the entries of 
$\biggl{\langle} B 
\begin{pmatrix} e_1 \\ e_2 \end{pmatrix} 
\biggr{\rangle}$ 
with the basis given by these entries 
is denoted by 
$B\cdot \mathfrak{M} _{\F}$. 
Note that $B\cdot \mathfrak{M} _{\F}$ depends on 
the choice of the basis of $\mathfrak{M} _{\F}$. 
We can see that 
if $\mathfrak{M}_{\F} \sim A$ for 
$A \in GL_2 \bigl( \F((u)) \bigr)$ 
with respect to a given basis, 
then we have 
\[
 B\cdot \mathfrak{M} _{\F} \sim 
 \phi (B)AB^{-1} 
\]
with respect to the induced basis. 

\begin{lem}\label{equiv}
Suppose $\F'$ is a finite extension of $\F$, 
and $x \in \mathscr{GR}_{V_{\F},0}(\F')$ 
corresponds to 
$\mathfrak{M} _{\F'}$. 
Put 
$\mathfrak{M} _{\F',i} = 
 \begin{pmatrix}
 u^{s_i} & v_i \\ 0 & u^{t_i}
 \end{pmatrix} \cdot
 \mathfrak{M} _{\F'}$
for 
$1\leq i \leq 2$, $s_i ,t_i \in \mathbb{Z}$ and 
$v_i \in \F((u))$.
Assume $\mathfrak{M} _{\F',1}$ and $\mathfrak{M} _{\F',2}$ 
correspond to $x_1 , x_2 \in \mathscr{GR}_{V_{\F},0}(\F')$ 
respectively. 
Then $x_1 =x_2$ if and only if
\[
s_1 =s_2,\ t_1 =t_2 \textrm{ and } 
v_1 -v_2 \in u^{t_1} \F'[[u]].
\]
\end{lem}
\begin{proof}
The equality $x_1 =x_2$ is equivalent to 
the existence of $B \in GL_2 (\F'[[u]])$ 
such that 
\[
B 
 \begin{pmatrix}
  u^{s_1 } & v_1 \\ 0 & u^{t_1 } 
 \end{pmatrix} =
 \begin{pmatrix}
  u^{s_2 } & v_2 \\ 0 & u^{t_2 } 
 \end{pmatrix}. 
\] 
It is further equivalent to the condition that 
\[
 \begin{pmatrix}
 u^{s_2 -s_1 } 
 & v_2 u^{-t_1 } - u^{s_2 -s_1 -t_1 } v_1 \\
 0 & u^{t_2 -t_1 }
 \end{pmatrix} 
 \in GL_2 (\F'[[u]]). 
\]  
The last condition is equivalent to the desired condition. 
\end{proof}

\section{Main theorem}

\begin{thm}\label{mainthm}
Let $q$ be the cardinality of $\F$. 
Then we have 
\[
 |M(C_{\F} ,K)| = 
 \sum_{n \geq 0} (a_n +a_n ')q^n . 
\]
Here $a_n$ and $a_n '$ are defined as 
in the introduction. 
\end{thm}

\begin{proof}
Since $V_{\F}$ is the trivial representation, 
$M_{\F} \sim 
\begin{pmatrix}
 1 & 0 \\ 0 & 1 
\end{pmatrix}$ 
for some basis. 
Let $\mathfrak{M} _{\F,0}$ be the lattice of 
$M_{\F}$ generated by the basis giving 
$M_{\F} \sim 
\begin{pmatrix}
 1 & 0 \\ 0 & 1 
\end{pmatrix}$.
By the Iwasawa decomposition, 
any sublattice of $M_{\F}$ can be written as 
$\begin{pmatrix}
  u^s & v \\ 0 & u^t 
 \end{pmatrix} \cdot 
 \mathfrak{M} _{\F,0}$
for $s,t \in \mathbb{Z}$ and $v \in \F((u))$. 
We put 
\[
 \mathscr{GR}_{V_{\F},0,s,t}(\F) =
 \biggl\{
 \begin{pmatrix}
  u^s & v \\ 0 & u^t 
 \end{pmatrix} \cdot 
 \mathfrak{M} _{\F,0}
 \in
 \mathscr{GR}_{V_{\F},0}(\F)
 \biggm| 
 v \in \F((u))
 \biggr\}. 
\]
Then
\[
 \mathscr{GR}_{V_{\F},0}(\F) =
 \bigcup_{s,t \in \mathbb{Z}}
 \mathscr{GR}_{V_{\F},0,s,t}(\F) 
\]
and this is a disjoint union by Lemma \ref{equiv}. 

We put 
\[
 \mathfrak{M} _{\F,s,t} 
 = \begin{pmatrix}
  u^s & 0 \\ 0 & u^t 
 \end{pmatrix} \cdot 
 \mathfrak{M} _{\F,0}.
\] 
Then we have 
$\mathfrak{M} _{\F,s,t} \sim 
 \begin{pmatrix}
  u^{(p-1)s} & 0 \\ 0 & u^{(p-1)t} 
 \end{pmatrix}$ 
with respect to the basis induced from 
$\mathfrak{M} _{\F,0}$. 
Any $\mathfrak{M}_{\F}$ in 
$\mathscr{GR}_{V_{\F},0,s,t}(\F)$ 
can be written as 
$\begin{pmatrix}
  1 & v \\ 0 & 1 
 \end{pmatrix} \cdot 
 \mathfrak{M} _{\F,s,t}$ 
for $v$ in $\F((u))$. 
Then we have 
\[
 \mathfrak{M} _{\F} \sim 
 \begin{pmatrix}
  u^{(p-1)s} & -vu^{(p-1)s} +\phi (v)u^{(p-1)t} 
  \\ 0 & u^{(p-1)t} 
 \end{pmatrix}
\]
with respect to the induced basis. 
The condition 
$u^e \mathfrak{M}_{\F} \subset 
(1\otimes \phi ) \bigl( \phi ^* (\mathfrak{M}_{\F}) \bigr)
\subset \mathfrak{M}_{\F}$ 
is equivalent to the following:
\begin{align*}
 0 \leq (p-1)s \leq e,\ 
 &0 \leq (p-1)t \leq e,\\ 
 &v_u (vu^{(p-1)s} -\phi (v)u^{(p-1)t}) \geq 
 \max \bigl\{ 0,(p-1)(s+t)-e \bigr\}.
\end{align*}
Conversely, $s,t \in \mathbb{Z}$ and 
$v \in \F((u))$ satisfying this condition gives 
a point of $\mathscr{GR}_{V_{\F},0,s,t}(\F)$ 
as 
$\begin{pmatrix}
  1 & v \\ 0 & 1 
 \end{pmatrix} \cdot 
 \mathfrak{M} _{\F,s,t}$. 
We put $r=-v_u (v)$. 

We fix 
$s,t \in \mathbb{Z}$ such that 
$0 \leq s,t \leq e_0$. 
The lowest degree term of 
$vu^{(p-1)s}$ is equal to 
that of $\phi (v)u^{(p-1)t}$ 
if and only if 
$v_u (v)=s-t$, 
in which case 
$v_u (vu^{(p-1)s})=ps-t$. 

In the case where $ps-t \geq \max \bigl\{ 0,(p-1)(s+t)-e \bigr\}$, 
the condition 
$v_u (vu^{(p-1)s} -\phi (v)u^{(p-1)t}) \geq 
 \max \bigl\{ 0,(p-1)(s+t)-e \bigr\}$ 
is equivalent to 
\[
 \min \bigl\{ v_u (vu^{(p-1)s}), 
 v_u\bigl( \phi (v)u^{(p-1)t} \bigr) 
 \bigr\} \geq 
 \max \bigl\{ 0,(p-1)(s+t)-e \bigr\}, 
\]
and further equivalent to 
\[
 r \leq \min \biggl\{
 (p-1)s, \frac{e-(p-1)s}{p}, 
 e-(p-1)t, \frac{(p-1)t}{p} 
 \biggr\}. 
\]
We put 
\[
 r_{s,t} = \min \Biggl\{
 (p-1)s, \biggl[ \frac{e-(p-1)s}{p} \biggr], 
 e-(p-1)t, \biggl[ \frac{(p-1)t}{p} \biggr] 
 \Biggr\}. 
\]
In this case, the number of the points of 
$\mathscr{GR}_{V_{\F},0,s,t}(\F)$ is 
equal to $q^{r_{s,t}}$ 
by Lemma \ref{equiv}. 

Next, we consider the case where 
$ps-t < \max \bigl\{ 0,(p-1)(s+t)-e \bigr\}$. 
We claim that the condition 
$v_u (vu^{(p-1)s} -\phi (v)u^{(p-1)t}) \geq 
 \max \bigl\{ 0,(p-1)(s+t)-e \bigr\}$ 
is satisfied if and only if 
\[
 v=\alpha u^{s-t} +v_+ 
 \textrm{ for } 
 \alpha \in \F 
 \textrm{ and } 
 v_+ \in \F((u)) 
 \textrm{ such that } 
 -v_u (v_+ ) \leq r_{s,t}. 
\]
Clearly, the latter implies the former. 
We prove the converse. 
We assume that the former condition. 
If 
\[
 \min \bigl\{ v_u (vu^{(p-1)s}), 
 v_u\bigl( \phi (v)u^{(p-1)t} \bigr) 
 \bigr\} \geq 
 \max \bigl\{ 0,(p-1)(s+t)-e \bigr\}, 
\]
we may take $\alpha=0$. 
So we may assume that 
\[
 \min \bigl\{ v_u (vu^{(p-1)s}), 
 v_u\bigl( \phi (v)u^{(p-1)t} \bigr) 
 \bigr\} < 
 \max \bigl\{ 0,(p-1)(s+t)-e \bigr\}. 
\]
Then the lowest degree term of 
$vu^{(p-1)s}$ is equal to 
that of $\phi (v)u^{(p-1)t}$, 
and the lowest degree term of 
$v$ can be writen as 
$\alpha u^{s-t}$ for $\alpha \in \F^{\times}$. 
We put $v_{+} =v-\alpha u^{s-t}$. 
We can see $-v_u (v_+ ) \leq r_{s,t}$, 
because 
$v_u (v_{+} u^{(p-1)s} -\phi (v_{+} )u^{(p-1)t}) \geq 
 \max \bigl\{ 0,(p-1)(s+t)-e \bigr\}$ and 
the lowest degree term of 
$v_{+} u^{(p-1)s}$ can not be equal to 
that of $\phi (v_{+} )u^{(p-1)t}$. 
Thus the claim has been proved, 
and the number of the points of 
$\mathscr{GR}_{V_{\F},0,s,t}(\F)$ is 
equal to $q^{r_{s,t} +1}$ 
by Lemma \ref{equiv}. 

We put 
$h_{s,t} =\log_q |\mathscr{GR}_{V_{\F},0,s,t}(\F)|$. 
Collecting the above results, 
we get the followings:
\begin{itemize}
 \item 
 If $s+t \leq e_0$ and $ps-t \geq 0$, 
 then $h_{s,t} =[(p-1)t/p]$. 
 \item
 If $s+t \leq e_0$ and $ps-t < 0$, 
 then $h_{s,t} =(p-1)s +1$. 
 \item
 If $s+t > e_0$ and 
 $ps-t \geq (p-1)(s+t) -e$, 
 then 
 $h_{s,t} =
 \bigl[ \bigl( e-(p-1)s \bigr) \big/ p \bigr]$. 
 \item
 If $s+t > e_0$ and 
 $ps-t < (p-1)(s+t) -e$, 
 then 
 $h_{s,t} = e-(p-1)t+1$. 
\end{itemize}

Now we have 
\[
 |M(C_{\F} ,K)|=
 \sum_{0 \leq s,t \leq e_0} 
 q^{h_{s,t}}. 
\]
We put 
\[
 S_n =\bigl\{ (s,t) \in \mathbb{Z}^2 \bigm| 
 0 \leq s,t \leq e_0 ,\ 
 h_{s,t} =n
 \bigr\},
\] 
and 
\begin{align*}
 S_{n,1} &=\bigl\{ (s,t) \in S_n \bigm|  
 s+t \leq e_0 ,\ 
 ps-t \geq 0
 \bigr\},\\
 S_{n,2} &=\bigl\{ (s,t) \in S_n \bigm|  
 s+t \leq e_0 ,\ 
 ps-t < 0
 \bigr\},\\
 S_{n,1} '&=\bigl\{ (s,t) \in S_n \bigm|  
 s+t > e_0 ,\ 
 ps-t \geq (p-1)(s+t) -e
 \bigr\},\\
 S_{n,2} '&=\bigl\{ (s,t) \in S_n \bigm|  
 s+t > e_0 ,\ 
 ps-t < (p-1)(s+t) -e
 \bigr\}. 
\end{align*}
It suffices to show that 
$|S_{n,1} |+|S_{n,2} | =a_n$ and 
$|S_{n,1} '|+|S_{n,2} '| =a_n '$.

Firstly, we calculate $|S_{n,1}|$.  
We assume $(s,t) \in S_{n,1}$. 
In the case $n_1 \neq 0$, 
we have $t=pn_0 +n_1 +1$ by 
$[(p-1)t/p]=(p-1)n_0 +n_1$. 
Then $ps \geq t =pn_0 +n_1 +1$ implies 
$s \geq n_0 +1$, and 
we have 
\[
 n_0 +1 \leq s \leq e_0 -pn_0 -n_1 -1.
\]
We note that if $t>e_0$, 
we have 
\[
 (e_0 -pn_0 -n_1 -1)-(n_0 +1)+1 =
 e_0 -(p+1)n_0 -n_1 -1<0. 
\] 
So we get 
\[
 |S_{n,1} | = 
 \max \{ e_0 -(p+1)n_0 -n_1 -1,0 \}. 
\]
In the case $n_1 =0$, 
we have $t=pn_0$ or $t=pn_0 +1$ by 
$[(p-1)t/p]=(p-1)n_0$. 
If $t=pn_0$, 
we have $n_0  \leq s \leq e_0 -pn_0$. 
If $t=pn_0 +1$, 
we have $n_0 +1 \leq s \leq e_0 -pn_0 -1$. 
So we get 
\[
 |S_{n,1} | = 
 \max \{ e_0 -(p+1)n_0 +1,0 \} + 
 \max \{ e_0 -(p+1)n_0 -1,0 \}. 
\]

Secondly, we calculate $|S_{n,2} |$. 
In the case $n_1 \neq 1$, 
we have $S_{n,2} =\emptyset$. 
In the case $n_1 =1$, 
we assume $(s,t) \in S_{n,2}$. 
Then $s=n_0$, and we have 
$pn_0 +1 \leq t \leq e_0 -n_0$. 
So we get 
\[
 |S_{n,2} |=
 \max \{ e_0 -(p+1)n_0 ,0 \}. 
\]
Collecting these results, 
we have $|S_{n,1} |+|S_{n,2} | =a_n$. 

Next, we calculate $|S_{n,1} '|$.  
We assume $(s,t) \in S_{n,1} '$. 
In the case $n_1 '\neq 0$, 
we have $s=e_0 -e_1 -pn_0 '-n_1 '-1$ by 
$\bigl[ \bigl( e-(p-1)s \bigr)
 \big/ p \bigr]=(p-1)n_0 '+n_1 ' +e_1$. 
We note that 
$\bigl[ \bigl( e-(p-1)s \bigr)
 \big/ p \bigr] =n\geq 0$ 
shows $s \leq e_0$. 
Then $ps-t \geq (p-1)(s+t) -e$ implies 
$pt \leq pe_0 -pn_0 ' -n_1 '-1$, and 
further implies 
$t \leq e_0 -n_0 ' -1$. 
So we have 
\[
 e_1 +pn_0 '+n_1 '+2 \leq t \leq e_0 -n_0 '-1.
\]
We note that 
$e_1 +pn_0 '+n_1 '+2=n+n_0 '+2 \geq 1$ 
and $e_0 -n_0 ' -1 \leq e_0$, 
because $n_0 ' \geq -1$. 
We note also that if $s<0$, 
then 
\[
 (e_0 -n_0 '-1)-(e_1 +pn_0 '+n_1 '+2)+1 =
 e_0 -e_1 -(p+1)n_0 '-n_1 '-2<0. 
\] 
So we get 
\[
 |S_{n,1} | = 
 \max \{ e_0 -e_1 -(p+1)n_0 '-n_1 '-2,0 \}. 
\]
In the case $n_1 '=0$, 
we have $s=e_0 -e_1 -pn_0 '-1$ 
or $s=e_0 -e_1 -pn_0 '$ by 
$\bigl[ \bigl( e-(p-1)s \bigr)
 \big/ p \bigr]=(p-1)n_0 '+e_1$. 
If $s=e_0 -e_1 -pn_0 '-1$, 
we have $e_1 +pn_0 '+2 \leq t \leq e_0 -n_0 '-1$. 
If $s=e_0 -e_1 -pn_0 '$, 
we have $e_1 +pn_0 '+1 \leq t \leq e_0 -n_0 '$. 
We note that $n_0 '\geq 0$, because $n_1 '=0$. 
So we get 
\[
 |S_{n,1} '| = 
 \max \{ e_0 -e_1 -(p+1)n_0 '-2,0 \} + 
 \max \{ e_0 -e_1 -(p+1)n_0 ',0 \}. 
\]

At last, we calculate $|S_{n,2} '|$. 
In the case $n_1 '\neq 1$, 
we have $S_{n,2} '=\emptyset$. 
In the case $n_1 '=1$, 
we assume $(s,t) \in S_{n,2} '$. 
Then $t=e_0 -n_0 '$, and we have 
$n_0 '+1 \leq s \leq e_0 -e_1 -pn_0 '-1$. 
Here we need some care, 
because there is the case $n_0 '=-1$, 
in which case $t >e_0$. 
Now $n_0 '=-1$ is equivalent to 
$n=0$ and $e_1 =p-2$. 
So we get 
\[
 |S_{n,2} '|=
 \max \{ e_0 -e_1 -(p+1)n_0 '-1,0 \} 
\]
except in the case where 
$n=0$ and $e_1 =p-2$, 
in which case $S_{n,2} '=\emptyset$.
Collecting these results, 
we have $|S_{n,1} '|+|S_{n,2} '| =a_n '$. 
This completes the proof. 
\end{proof}

\begin{eg}
If $K=\mathbb{Q}_p (\zeta_p)$ 
and $\F =\F_p$, 
we have 
$\bigl| M \bigl( C_{\F_p} ,
 \mathbb{Q}_p (\zeta_p )\bigr) \bigr| =p+3$ 
by Theorem \ref{mainthm}. 
We know that 
$\mathbb{Z}/p\mathbb{Z} \oplus \mathbb{Z}/p\mathbb{Z}$, 
$\mathbb{Z}/p\mathbb{Z} \oplus \mu_p$ and 
$\mu_p \oplus \mu_p$ over 
$\mathcal{O}_{\mathbb{Q}_p (\zeta_p)}$ 
have the generic fibers 
that are isomorphic to $C_{\F_p}$. 
We can see 
$\bigl| \Aut (C_{\F_p} ) \bigr| =p(p+1)(p-1)^2$.
On the other hand, 
we have 
\[
 \Aut (\mathbb{Z}/p\mathbb{Z} \oplus \mu_p ) 
 \cong \Aut (\mathbb{Z}/p\mathbb{Z} ) \times 
 \Hom (\mathbb{Z}/p\mathbb{Z} ,\mu_p) \times 
 \Aut (\mu_p ), 
\]
because 
$\Hom (\mu_p ,\mathbb{Z}/p\mathbb{Z} ) =0$. 
In particular, we have 
$\bigl| \Aut (\mathbb{Z}/p\mathbb{Z} \oplus \mu_p ) \bigr| =p(p-1)^2$. 
Hence, there are $(p+1)$-choices of an isomorphism 
$C_{\F_p} \xrightarrow{\sim} 
 \bigl( \mathbb{Z}/p\mathbb{Z} \oplus \mu_p 
 \bigr)_{\mathbb{Q}_p (\zeta_p )}$ 
that give the different elements of 
$M \bigl( C_{\F_p} ,\mathbb{Q}_p (\zeta_p )\bigr)$. 
So the equation 
$\bigl| M \bigl( C_{\F_p} ,
 \mathbb{Q}_p (\zeta_p )\bigr) \bigr|=1+(p+1)+1$ 
shows that there does not exist 
any other isomorphism class of 
finite flat models of $C_{\F_p}$. 
\end{eg}

\begin{rem}
Theorem \ref{mainthm} is 
equivalent to an explicit calculation of 
the zeta function of $\mathscr{GR}_{V_{\F},0}$, 
and we can see that 
$\dim \mathscr{GR}_{V_{\F},0} = 
 \max \bigl\{ n \geq 0 \bigm| 
 a_n +a_n ' \neq 0 \bigr\}$. 
\end{rem}

\end{document}